\documentclass[10pt]{amsart}
\author[H.~Kawanoue]{Hiraku Kawanoue} 
\address{
College of Science and Engineering, Chubu University
\newline 
Matsumoto-cho, Kasugai-shi, 
Aichi 487-8501, JAPAN
} 
\email{kawanoue@kurims.kyoto-u.ac.jp} 
\title{On a conjecture of Feigin, Wang and Yoshinaga}
\newcommand{\ds}{\displaystyle}
\newcommand{\bZ}{\mathbb{Z}}

\newtheorem*{conj}{Conjecture}
\newtheorem*{thm}{Theorem}
\newtheorem{lem}{Lemma}

\newtheorem{prop}{Proposition}
\usepackage{version,graphics,url,txfonts,times}
\begin{document}
\begin{abstract}
We settled a conjecture 
of Feigin, Wang and Yoshinaga, 
appeared in the preprint 
``Integral expressions for derivations of multiarrangements''.

\end{abstract}
\maketitle

This is a short report about a conjecture 
of Feigin, Wang and Yoshinaga in \cite{FWY}.

\bigskip

In the preprint \cite{FWY}, 
Feigin, Wang and Yoshinaga constructed a candidate of basis 
for the module of logarithmic vector fields of 
the extended Catalan arrangement 
$\operatorname{Cat}(B_2,m)$ of type $B_2$
in the following way. First 
they define 
$${f_i^m}(x,y)
=\int_0^xt^{2i}(t^2-x^2)^m(t^2-y^2)^m\,dt
\qquad(i,m\in\bZ_{\geq0}).
$$
Set $p=2m+i$.  Then ${f_i^m}(x,y)$ is expressed as 
$${f_i^m}(x,y)=\sum_{0\leq k\leq m}c_{m,i,k}x^{2p-2k+1}y^{2k}.$$
Next, they define $\widetilde{f_i^m}(x,y)$ by deforming ${f_i^m}(x,y)$ 
as 
$$
\widetilde{f_i^m}(x,y)=
\sum_{0\leq k\leq m}c_{m,i,k}
(x+p-k)_{2p-2k+1}(y+p-m)_k
(y+m-p+k-1)_k,
$$
where $(\alpha)_k=\alpha\cdot(\alpha-1)\cdot\dotsb\cdot(\alpha-k+1)$.
Using this $\widetilde{f_i^m}(x,y)$, they define 
$$
\widetilde{\eta_i^m}=
\widetilde{f_i^m}(x,y)\partial_x+
\widetilde{f_i^m}(y,x)\partial_y.
$$

\bigskip

Note that if 
$\widetilde{\eta_0^m}, \widetilde{\eta_1^m}\in D(\operatorname{Cat}(B_2,m))$,
then they form a basis for $D(\operatorname{Cat}(B_2,m))$.  Actually, 
since $\deg\widetilde{\eta_i^m}=4m+2i+1$ and 
$\operatorname{Cat}(B_2,m))$ is defined by 
$$\varphi=(x+m)_{2m+1}(y+m)_{2m+1}(x+y+m)_{2m+1}(x-y+m)_{2m+1}=0,$$
there exists a constant $C$ such that 
the Saito determinant is of the form 
$$
|\operatorname{M}(\widetilde{\eta_0^m}, \widetilde{\eta_1^m})|
=\widetilde{f_0^m}(x,y) \widetilde{f_1^m}(y,x) 
- \widetilde{f_0^m}(y,x)  \widetilde{f_1^m}(x,y) 
=C\varphi.
$$
Comparing the coefficients of $x^{6m+3}y^{2m+1}$ in each side, we have 
$$
C=-c_{m,0,m}c_{m,1,0}
=-(-1)^m\int_0^1(s^2-1)^mds
\cdot\int_0^1s^{2m+2}(s^2-1)^mds
\neq0.
$$
Thus, by Saito's criterion, 
$\widetilde{\eta_0^m}$ and $\widetilde{\eta_1^m}$ form a basis 
for $D(\operatorname{Cat}(B_2,m))$. 

\bigskip

Since 
$\widetilde{f_i^m}(x,y)\in((x+p-m)_{2p-2m+1})\subset((x+m)_{2m+1})$, 
it is clear that 
$$
\widetilde{\eta_i^m}(x)\in((x+m)_{2m+1}),\quad
\widetilde{\eta_i^m}(y)\in((y+m)_{2m+1}).$$
Suppose we know that 
$$
\widetilde{\eta_i^m}(x+y)\in((x+y+m)_{2m+1}). 
$$
Then, since 
$\widetilde{f_i^m}(y,x)$ is odd in $x$ and even in $y$, 
it follows that 
$$
\widetilde{\eta_i^m}(x-y)
=\widetilde{f_i^m}(x,y)-\widetilde{f_i^m}(y,x)
=\widetilde{f_i^m}(x,-y)+\widetilde{f_i^m}(-y,x)
\in((x-y+m)_{2m+1}).
$$
Therefore the following conjecture 
of Feigin, Wang and Yoshinaga \cite{FWY} 
implies that 
$\widetilde{\eta_0^m}$ and $\widetilde{\eta_1^m}$ form a basis 
for $D(\operatorname{Cat}(B_2,m))$. 
\begin{conj}[Conjecture 5.5 in \cite{FWY}]
For $i,m\in\bZ_{\geq0}$, we have 
$$ \widetilde{f_i^m}(x+y)+\widetilde{f_i^m}(x+y)
\in((x+y+m)_{2m+1}).$$
 \end{conj}

We give an affirmative answer to this conjecture.

\bigskip

\noindent(1)\quad We write down the expansion of 
$\widetilde{f_i^m}(x,y)$ explicitly.  We have 
\begin{align*}
\lefteqn{
{f_i^m}(x,y)
=\int_0^xt^{2i}(t^2-x^2)^m(t^2-y^2)^m\,dt
}
\\
&=x^{4m+2i+1}\int_0^1s^{i}(s-1)^m(s-(y/x)^2)^m\,\dfrac{ds}{2\sqrt{s}}
\\
&=\dfrac12x^{4m+2i+1}\sum_{0\leq u\leq m}
\binom{m}{u}
(-1)^{u}(y/x)^{2(m-u)}
\int_0^1
s^{u+i-1/2}
(1-s)^m
\,ds
\\
&=\dfrac12\sum_{0\leq u\leq m}
\binom{m}{u}
(-1)^{u}
x^{2m+2i+2u+1}
y^{2m-2u}
B(u+i+1/2,m+1)
\\
&=\dfrac12
\sum_{0\leq u\leq m}
\binom{m}{u}
(-1)^{u}
x^{2m+2i+2u+1}
y^{2m-2u}
\dfrac{m!}{(m+i+u+1/2)_{m+1}}.
\end{align*}
By definition of 
$\widetilde{f_i^m}(x,y)$, it follows that 
$
2\widetilde{f_i^m}(x,y)=
\sum_{0\leq u\leq m}
F_{i,m,u}
$,
where 
\begin{align*}
F_{i,m,u}&=
\binom{m}{u}
\dfrac{(-1)^{u}m!}{(m+i+u+1/2)_{m+1}}
\\&
\phantom{=}
(x+m+i+u)_{2m+2i+2u+1}
(y+m+i)_{m-u}
(y-i-u-1)_{m-u}.
\end{align*}

\bigskip

\noindent(2)\quad
We construct a recurrence formula for 
$\widetilde{f_i^m}(x,y)$.  For this purpose, 
we denote a partial sum of $\widetilde{f_i^m}(x,y)$ 
as 
$S_{i,m,l}=\sum_{l\leq u\leq m}F_{i,m,u}$,
and set 
$$
M_{i, m, l}= -2S_{i + 1, m, l} + 
(x^2 + y^2 - (i + m + 1)^2 - i^2)S_{i, m, l}
- \dfrac{2i - 1}{2m + 2}S_{i - 1, m + 1, l + 1}.
$$

\begin{lem}\label{lem1}
\begin{align*}
M_{i, m, l}
&=\binom{m}{l}
\dfrac{(-1)^lm!}{(m + i + l + 1/2)_{m + 1}}
\{y^2 + l( 2m + 2i + l+2) - i^2\}
\\&\phantom{=}
(x + m + i + l)_{ 2m + 2i + 2l+1}
(y + m + i)_{m - l}(y - i - l - 1)_{m - l}
.
\end{align*}
\end{lem}

\medskip

\begin{proof}
Proof by descending induction on $l$. 

The case $l=m+1$ is clear since 
$S_{\ast,m,m+1}=S_{\ast,m+1,m+2}=
\binom{m}{m+1}=0$.

Let $l\leq m$ and assume the formula holds for $M_{i,m,l+1}$.
Since 
$$
M_{i, m, l}=
-2F_{i + 1, m, l} + 
(x^2 + y^2 - (i + m + 1)^2 - i^2)F_{i, m, l}
- \dfrac{2i - 1}{2m + 2}F_{i - 1, m + 1, l + 1}
+M_{i, m, l+1},
$$
we have 
\begin{align*}
&M_{i, m, l}
=
-2
\binom{m}{l}
\dfrac{(-1)^{l}m!}{(m+i+l+3/2)_{m+1}}
(x+m+i+l+1)_{2m+2i+2l+3}
(y+m+i+1)_{m-l}
\\&
(y-i-l-2)_{m-l}
+ 
(x^2 + y^2 - (i + m + 1)^2 - i^2)
\binom{m}{l}
\dfrac{(-1)^{l}m!}{(m+i+l+1/2)_{m+1}}
\\&
(x+m+i+l)_{2m+2i+2l+1}
(y+m+i)_{m-l}
(y-i-l-1)_{m-l}
- \dfrac{2i - 1}{2m + 2}
\binom{m+1}{l+1}
\\&
\dfrac{(-1)^{l+1}(m+1)!}{(m+i+l+3/2)_{m+2}}
(x+m+i+l+1)_{2m+2i+2l+3}
(y+m+i)_{m-l}
(y-i-l-1)_{m-l}
\\&
+
\binom{m}{l+1}
\dfrac{(-1)^{l+1}m!}{(m + i + l + 3/2)_{m + 1}}
(y^2 + (l+1)( 2m + 2i + l+ 3) - i^2)\cdot
\\&
(x + m + i+ l +1)_{ 2m + 2i + 2l+3}
(y + m + i)_{m - l-1}(y - i - l  - 2)_{m - l-1}
\\
&
=\binom{m}{l}\dfrac{(-1)^{l}m!}{(m+i+l+3/2)_{m+2}}
(x+m+i+l)_{2m+2i+2l+1}(y+m+i)_{m-l-1}(y-i-l-2)_{m-l-1}
\\&
\Bigl[
-2(i+l+1/2)\{x^2-(m+i+l+1)^2\}\{y^2-(m+i+1)^2\}
\\&
+ (m+i+l+3/2)
(x^2 + y^2 - (i + m + 1)^2 - i^2)\{y^2-(l+i+1)^2\}
\\&
+ \dfrac{2i - 1}{2}\dfrac{m+1}{l+1}
\{x^2-(m+i+l+1)^2\}\{y^2-(l+i+1)^2\}
\\&
-
\dfrac{m-l}{l+1}(i+l+1/2)(y^2 + (l+1)( 2m + 2i + l+ 3) - i^2)\cdot
\{x^2-(m+i+l+1)^2\}
\Bigr].
\end{align*}
One can show that 
$$\Bigl[\cdots\Bigr]=
(m + i + l + 3/2)\{y^2-(i+l+1)^2\}
(y^2 + l(2m+2i+l+2)-i^2 ),
$$
which verifies the formula for $M_{i,m,l}$.
\end{proof}
\begin{prop}\label{prop1}
$$
\dfrac{2i-1}{2m+2}
\widetilde{f_{i-1}^{m+1}}(x,y)
= 
\Bigl\{x^2 +y^2 -(i+m+1)^2 -i^2\Bigr\}
\widetilde{f_i^m}(x,y)
-2
\widetilde{f_{i+1}^m}(x,y)
.$$
\end{prop}

\medskip

\begin{proof}
By applying Lemma \ref{lem1} with $l=0$, we have 
$$
M_{i, m, 0}
=
\dfrac{m!}{(m + i  + 1/2)_{m + 1}}
(x + m + i )_{ 2m + 2i +1}
(y + m + i)_{m+1 }(y - i)_{m+1 }.
$$
On the other hand, we have 
$$
F_{i-1,m+1,0}=
\dfrac{(m+1)!}{(m+i+1/2)_{m+2}}
(x+m+i)_{2m+2i+1}
(y+m+i)_{m+1}
(y-i)_{m+1}.
$$
It follows that 
$
\dfrac{2i-1}{2m+2}F_{i-1,m+1,0}=M_{i, m, 0}
$, 
and hence 
\begin{align*}
\dfrac{2i-1}{2m+2}S_{i-1,m+1,0}
&=M_{i, m, 0}+\dfrac{2i-1}{2m+2}S_{i-1,m+1,1}
\\
&=
-2S_{i + 1, m, 0} + \Bigl\{
x^2 + y^2 - (i + m + 1)^2 - i^2\Bigr\}S_{i, m, 0}.
\end{align*}
Now the assertion is clear since 
$S_{i,m,0}=2\widetilde{f_{i}^m}(x,y)$.
\end{proof}

\bigskip

\noindent(3)\quad
We determine $\widetilde{f_{i}^m}(x,y)$ when 
$x$ takes a negative half-integral value in Proposition \ref{prop2}.
For this purpose, we preapare 2 lemmata. 
Lemma \ref{lem2} is for the case $m=0$, 
while 
Lemma \ref{lem3} is for the recurrence formula of this value. 

\begin{lem}\label{lem2}
\quad
$\ds 
\sum_{t=a}^\infty\dfrac{(b+t)_{2t}}{(z+t)_{2t+2}}
=
\dfrac{(b+a)_{2a}}{(z+b)(z+a-1)_{2a}(z-b-1)}
\quad(a,b\in\bZ_{\geq0})$.
\end{lem}
\begin{proof}
Proof by descending induction on $a$.  

First consider the case $a>b$.  Then, we have 
$(b+t)_{2t}=0$ for any $t\geq a$.  Therefore 
both sides of the asserted formula is $0$.

Assume that the assertion holds for $a+1$. 
By the induction hypothesis, we have 
\begin{align*}
&\sum_{t=a}^\infty\dfrac{(b+t)_{2t}}{(z+t)_{2t+2}}
= 
\dfrac{(b+a)_{2a}}{(z+a)_{2a+2}}+
\dfrac{(b+a+1)_{2a+2}}{(z+b)(z+a)_{2a+2}(z-b-1)}
\\&
= 
\dfrac{(b+a)_{2a}}{(z+b)(z+a)_{2a+2}(z-b-1)}
\Bigl\{(z+b)(z-b-1)+(b+a+1)(b-a)\Bigr\}
\\&
= \dfrac{(b+a)_{2a}(z+a )(z-a - 1)}{(z+b)(z+a)_{2a+2}(z-b-1)}
= \dfrac{(b+a)_{2a}}{(z+b)(z+a-1)_{2a}(z-b-1)}.
\qedhere
\end{align*}
\end{proof}

\bigskip

For $i\in\bZ$ and $m,k,l\in\bZ_{\geq0}$, we define 
$$
N_{i, m, k,l}= -2U_{i + 1, m, k,l} + 
\Bigl\{
\Bigl(k+\dfrac12\Bigr)^2 + y^2 - (i + m + 1)^2 - i^2\Bigr\}U_{i, m, k,l}
- \dfrac{2i - 1}{2m + 2}U_{i - 1, m + 1, k,l }.
$$
where 
$U_{i,m,k,l}=\sum_{t=l}^{k} G_{i,m,k,t}$ and 
\begin{align*}
&G_{i,m,k,t}=
-(i - 1/2)_{2i + m-k}
(y +m-k-1/2 )_{2m-2k}
(m+t)_m(k+t)_{2t}
\\ &\phantom{=-}
(i+2m+t+1/2)_{t}
(i + m+k+1/2)_{k-t}
(y + m+k+1/2)_{k-t}
(y - m-t-3/2)_{k-t}.
\end{align*}

\begin{lem}\label{lem3}
\begin{align*}
N_{i,m,k,l}&=
\dfrac{(m+l)_{m+1}}{m+1} 
(k + l)_{2l}
(i - 1/2)_{2i + m - k}
( i + m+k + 1/2)_{k - l}
(i + 2m + l+1/2)_{l-1}
\\ &\phantom{=}
(y + m + k + 1/2)_{ k - l}
(y + m - k - 1/2)_{2m - 2k}
(y - m - l-3/2)_{ k - l}
\\ &\phantom{=}
\Bigl\{( y^2-i^2 )(i + 3m + l+5/2) 
+ (i + m + l+1/2)(i + m - k + 1/2)(i + m + k + 3/2)
\Bigr\}.
\end{align*}
\end{lem}
\begin{proof}
Proof by descending induction on $l$. 
First consider the case $l>k$. Then, 
we have $U_{\ast,\ast,k,l}=0$ and hence 
$N_{i,m,k,l}=0$.  We also have $(k+l)_{2l}=0$.  
Therefore both sides of the asserted formula is $0$. 

Assume that the assertion holds for $l+1$.  
By  definition, we have 
\begin{align*}
N_{i,m,k,l}&=
N_{i,m,k,l+1}
+
\\&\phantom{=}
\Bigl[
-2G_{i + 1, m, k,l} + 
\Bigl\{
\Bigl(k+\dfrac12\Bigr)^2 + y^2 - (i + m + 1)^2 - i^2\Bigr\}G_{i, m, k,l}
- \dfrac{2i - 1}{2m + 2}G_{i - 1, m + 1, k,l }
\Bigr].
\end{align*}
By the induction hypothesis, it follows that 
\begin{align*}
&N_{i,m,k,l}=
\dfrac{(m+l+1)_{m+1}}{m+1} 
(k + l+1)_{2l+2}
(i - 1/2)_{2i + m - k}
( i + m+k + 1/2)_{k - l-1}
(i + 2m + l+3/2)_{l}
\\ &
(y + m + k + 1/2)_{ k - l-1}
(y + m - k - 1/2)_{2m - 2k}
(y - m - l-5/2)_{ k - l-1}
\\ &
\Bigl\{( y^2-i^2 )(i + 3m + l+7/2) 
+ (i + m + l+3/2)(i + m - k + 1/2)(i + m + k + 3/2)
\Bigr\}
\\ &
-2
\Bigl[
-(i + 1/2)_{2i + m-k+2}
(y +m-k-1/2 )_{2m-2k}
(m+l)_m(k+l)_{2l}
\\ &
(i+2m+l+3/2)_{l}
(i + m+k+3/2)_{k-l}
(y + m+k+1/2)_{k-l}
(y - m-l-3/2)_{k-l}
\Bigr]
\\ &
+ \Bigl\{\Bigl(k+\dfrac12\Bigr)^2 + y^2 - (i + m + 1)^2 - i^2\Bigr\}
\Bigl[
-(i - 1/2)_{2i + m-k}
(y +m-k-1/2 )_{2m-2k}
(m+l)_m(k+l)_{2t}
\\ &
(i+2m+l+1/2)_{l}
(i + m+k+1/2)_{k-l}
(y + m+k+1/2)_{k-l}
(y - m-l-3/2)_{k-l}
\Bigr]
\\&
- \dfrac{2i - 1}{2m + 2}
\Bigl[
-(i - 3/2)_{2i + m-k-1}
(y +m-k+1/2 )_{2m-2k+2}
(m+l+1)_{m+1}(k+l)_{2l}
\\ &
(i+2m+l+3/2)_{l}
(i + m+k+1/2)_{k-l}
(y + m+k+3/2)_{k-l}
(y - m-l-5/2)_{k-l}
\Bigr]
\\&
=
\dfrac{(m+l)_m}{m+1}(k+l)_{2l}
(i - 1/2)_{2i + m-k}
(i + m+k+1/2)_{k-l-1}
(i+2m+l+1/2)_{l-1}
\\ &
(y + m+k+1/2)_{k-l-1}
(y +m-k-1/2 )_{2m-2k}
(y - m-l-5/2)_{k-l-1}
\\ &
\Bigl[
(m+l+1)(k + l+1)(k - l)(i + 2m + l+3/2)
\Bigl\{( y^2-i^2 )(i + 3m + l+7/2) 
\\&
+ (i + m + l+3/2)(i + m - k + 1/2)(i + m + k + 3/2)
\Bigr\}
+2(m+1)(i + 1/2)(-i-m+k - 1/2)
\\ &
(i+2m+l+3/2)(i + m+k+3/2)(y + m+l+3/2)(y - m-l-3/2)
\\ & 
- 
\Bigl\{\Bigl(k+\dfrac12\Bigr)^2 + y^2 - (i + m + 1)^2 - i^2\Bigr\}
(m+1)(i+2m+3/2)(i + m+l+3/2)
\\&
(y + m+l+3/2)(y - m-l-3/2)+(m+l+1)(y +m-k+1/2 )(y -m+k-1/2 )
\\ &
(i+2m+l+3/2)(i + m+l+3/2)(y + m+k+3/2)(y - m-k-3/2)
\Bigr]
\\ &=
 \dfrac{(m+l)_m}{m+1}(k+l)_{2l}
(i - 1/2)_{2i + m-k}
(i + m+k+1/2)_{k-l-1}
(i+2m+l+1/2)_{l-1}
\\ &
(y + m+k+1/2)_{k-l-1}
(y +m-k-1/2 )_{2m-2k}
(y - m-l-5/2)_{k-l-1}
\\&
\Bigl[\ 
l( i + m+l + 3/2)
(y + m + l + 3/2)(y - m - l-3/2)
\\&
\Bigl\{( y^2-i^2 )(i + 3m + l+5/2) + (i + m + l+1/2)(i + m - k + 1/2)
(i + m + k + 3/2)\Bigr\}\Bigr]
\\&=
 \dfrac{(m+l)_{m+1}}{m+1}(k+l)_{2l}
(i - 1/2)_{2i + m-k}
(i + m+k+1/2)_{k-l}
(i+2m+l+1/2)_{l-1}
\\ &
(y + m+k+1/2)_{k-l}
(y +m-k-1/2 )_{2m-2k}
(y - m-l-3/2)_{k-l}
\\&
\Bigl\{( y^2-i^2 )(i + 3m + l+5/2) + (i + m + l+1/2)(i + m - k + 1/2)
(i + m + k + 3/2)\Bigr\},
\end{align*}
which verifies the assertion for $l$.
\end{proof}

\begin{prop}\label{prop2}
For $k\in\bZ_{\geq0}$, we have 
\begin{align*}
&2\widetilde{f_{i}^m}(-1/2-k,y)
=
-(i - 1/2)_{2i + m-k}
(y +m-k-1/2 )_{2m-2k}
\sum_{t=0}^{k} 
(m+t)_m(k+t)_{2t}
\\&
(i+2m+t+1/2)_{t}
(i + m+k+1/2)_{k-t}
(y + m+k+1/2)_{k-t}
(y - m-t-3/2)_{k-t}.
\end{align*}
\end{prop}
\begin{proof}
Notice that $({\rm RHS})=U_{i,m,k,0}$.
We denote $({\rm LHS})$ as $T_{i,m,k}$.  

\medskip

First consider the case $m=0$.  By definition, we have 
\begin{align*}
U_{i,0,k,0}&
=
-(i - 1/2)_{2i -k}
(y -k-1/2 )_{-2k}
\sum_{t=0}^{k} 
(k+t)_{2t}
\\&
(i+t+1/2)_{t}
(i +k+1/2)_{k-t}
(y +k+1/2)_{k-t}
(y -t-3/2)_{k-t}
\\&
=
-\dfrac{
(i +k+1/2)_{2i+1}
}{i+1/2}
\Bigl\{y^2-(k+1/2)^2\Bigr\}
\sum_{t=0}^{k} 
\dfrac{
(k+t)_{2t}
}{
(y +t+1/2)_{2t+2}
}.
\end{align*}
By applying Lemma \ref{lem2} with 
$a=0$, $b=k\geq0$ and $z=y+1/2$, 
we  have
\begin{align*}
U_{i,0,k,0}&=
-\dfrac{
(i +k+1/2)_{2i+1}
}{i+1/2}
\Bigl\{y^2-(k+1/2)^2\Bigr\}
\dfrac1{(y+k+1/2)(y-k-1/2)}
\\&=-\dfrac{
(i +k+1/2)_{2i+1}
}{i+1/2},
\end{align*}
and it follows that 
$$
T_{i,0,k}=F_{i,0,0}|_{x=-1/2-k}=\dfrac{(-1/2-k+i)_{2i+1}}{i+1/2}
=-\dfrac{(i +k+1/2)_{2i+1}}{i+1/2}=U_{i,0,k,0}.
$$

Next assume $m\geq1$. Since $T_{i,m,k}=S_{i,m,0}|_{x=-1/2-k}$, 
Proposition \ref{prop1} implies that 
$$
\dfrac{2i-1}{2m+2}
T_{i-1,m+1,k}
= 
\Bigl\{\Bigl(k+\dfrac12\Bigr)^2 +y^2 -(i+m+1)^2 -i^2\Bigr\}
T_{i,m,k}
-2
T_{i+1,m,k}
.$$
On the other hand, 
since $N_{i,m,k,0}=0$ by Lemma \ref{lem3}, we have 
$$
 -2U_{i + 1, m, k,0} + 
\Bigl\{
\Bigl(k+\dfrac12\Bigr)^2 + y^2 - (i + m + 1)^2 - i^2\Bigr\}U_{i, m, k,0}
- \dfrac{2i - 1}{2m + 2}U_{i - 1, m + 1, k,0}=0.
$$
Therefore, $T_{i,m,k}$ and $U_{i,m,k,0}$ satisfy the same recurrence formula, 
and they coincide when $m=0$.  Therefore we can conclude that 
$T_{i,m,k}=U_{i,m,k,0}$.
\end{proof}

\bigskip

\noindent(4)\quad
We determine the factors of $\widetilde{f_{i}^m}(x,y) \mod{(x+y\pm m)}$, 
and give the answer to the  conjecture.

\bigskip

\begin{prop}\label{prop3}
The following holds.
\item{\rm (a)}\quad
There exists a constant $A_{i,m}$ depending only on $i$ and $m$ such that 
\begin{align*}
2\widetilde{f_{i}^m}(x,y)
&\equiv
\phantom{-}
A_{i,m}
(x  + 2m + i)_{3m + 2i + 1}
(x + m - 1/2)_{m}
\\
&\equiv
-
A_{i,m}
(y  +2m+i)_{3m + 2i + 1}
(y  +m -1/2)_{m}
\mod{(x+y+m)}.
\end{align*}

\item{\rm (b)}\quad
There exists a constant $B_{i,m}$ depending only on $i$ and $m$ such that 
\begin{align*}
2\widetilde{f_{i}^m}(x,y)
&\equiv
\phantom{-}
B_{i,m}
(x  + m + i)_{3m + 2i + 1}
(x  - 1/2)_{m}
\\
&\equiv
-B_{i,m}
(y  + m + i)_{3m + 2i + 1}
(y  -1/2)_{m}
\mod{(x+y-m)}.
\end{align*}

\end{prop}

\pagebreak[4]

\begin{proof}
\item{(a)}\quad
Since $F_{i,m,u}$ is divided by 
\begin{align*}
&(x+m+i+u)_{2m+2i+2u+1}
(y-i-u-1)_{m-u}
\\&
=
(-1)^{m-u}
(x+m+i+u)_{2m+2i+2u+1}
(-y+i+m)_{m-u}
\\
&\equiv
(-1)^{m-u}
(x+m+i+u)_{2m+2i+2u+1}
(x+2m+i)_{m-u}
\mod{(x+y+m)}
\\&
=(-1)^{m-u}(x+2m+i)_{3m+2i+u+1}
=(-1)^{m-u}(x+2m+i)_{3m+2i+1}(x-m-i-1)_{u},
\end{align*}
we have 
$S_{i,m,0}\in((x+2m+i)_{3m+2i+1})+(x+y+m)$. 

Since Proposition \ref{prop2} guarantees 
$$2\widetilde{f_i^m}(-1/2-k,k-m+1/2)=0\quad(0\leq k<m),$$
we have 
$S_{i,m,0}\in((x+m-1/2)_{m})+(x+y+m)$.  
It follows that 
$$
S_{i,m,0}\in((x+2m+i)_{3m+2i+1}(x+m-1/2)_{m})+(x+y+m). 
$$
Since $\deg S_{i,m,0}\leq 4m+2i+1$, there exists a constant 
$A_{i,m}$ such that 
$$S_{i,m,0}\equiv A_{i,m}(x+2m+i)_{3m+2i+1}(x+m-1/2)_{m}\mod{(x+y+m)}.$$
It is straightforward that 
\begin{align*}
S_{i,m,0}&\equiv A_{i,m}(x+2m+i)_{3m+2i+1}(x+m-1/2)_{m}
\equiv A_{i,m}(-y+m+i)_{3m+2i+1}(-y-1/2)_{m}
\\&
\equiv A_{i,m}
(-1)^{3m+2i+1}(y+2m+i)_{3m+2i+1}
(-1)^m(y+m-1/2)_{m}
\\&
\equiv -A_{i,m}(y+2m+i)_{3m+2i+1}(y+m-1/2)_{m}
\mod{(x+y+m)}.
\end{align*}

\medskip

\item{(b)}\quad
Since $F_{i,m,u}$ is divided by 
\begin{align*}
&(x+m+i+u)_{2m+2i+2u+1}
(y+m+i)_{m-u}
\\&
=
(-1)^{m-u}
(x+m+i+u)_{2m+2i+2u+1}
(-y-u-i-1)_{m-u}
\\
&\equiv
(-1)^{m-u}
(x+m+i+u)_{2m+2i+2u+1}
(x-m-u-i-1)_{m-u}
\mod{(x+y-m)}
\\&
=(-1)^{m-u}(x+m+i+u)_{3m+2i+u+1}
=(-1)^{m-u}(x+m+i+u)_u(x+m+i)_{3m+2i+1},
\end{align*}
we have 
$S_{i,m,0}\in((x+m+i)_{3m+2i+1})+(x+y-m)$. 

Remember that $\widetilde{f_i^m}$ is an odd function in $x$, i.e., 
$\widetilde{f_i^m}(-x,y)=-\widetilde{f_i^m}(x,y)$, by construction.  
Therefore, by Proposition \ref{prop2}, we have 
$$2\widetilde{f_i^m}(k+1/2,m - k - 1/2)
=-2\widetilde{f_i^m}(-1/2-k,m - k - 1/2)=0\quad(0\leq k<m).$$
It follows that 
$S_{i,m,0}\in((x-1/2)_{m})+(x+y-m)$,
and hence 
$$
S_{i,m,0}\in((x+m+i)_{3m+2i+1}(x-1/2)_{m})+(x+y-m). 
$$
Since $\deg S_{i,m,0}\leq 4m+2i+1$, there exists a constant 
$B_{i,m}$ such that 
$$S_{i,m,0}\equiv B_{i,m}(x+m+i)_{3m+2i+1}(x-1/2)_{m}\mod{(x+y-m)}.$$
It is straightforward that 
\begin{align*}
S_{i,m,0}&\equiv B_{i,m}(x+m+i)_{3m+2i+1}(x-1/2)_{m}
\equiv B_{i,m}(-y+2m+i)_{3m+2i+1}(-y+m-1/2)_{m}
\\&
\equiv B_{i,m}
(-1)^{3m+2i+1}
(y+m+i)_{3m+2i+1}
(-1)^{m}
(y-1/2)_{m}
\\&
\equiv -B_{i,m}(y+m+i)_{3m+2i+1}(y-1/2)_{m}
\mod{(x+y-m)}
\qedhere
\end{align*}
\end{proof}

\pagebreak[4]

\begin{thm}[ $=$  Conjecture 5.5 in \cite{FWY}]
\qquad$\ds \widetilde{f_i^m}(x,y)+\widetilde{f_i^m}(y,x)
\in((x+y+m)_{2m+1})$.
\end{thm}
\begin{proof}
Set 
$V_{i,m}=\widetilde{f_i^m}(x,y)+\widetilde{f_i^m}(y,x)$.
By Proposition \ref{prop3}, it is clear that 
$$V_{i,m}\in(x+y\pm m).\eqno{(\ast)}$$
Proof by induction on $m$.  
The case $m=0$ is clear from $(\ast)$.  

Assume $m\geq1$. 
By Proposition \ref{prop1}, we have 
$$
\dfrac{2i+1}{2m}
V_{i,m}
= 
\Bigl\{x^2 +y^2 -(i+m+1)^2 -(i+1)^2\Bigr\}
V_{i+1,m-1}
-2
V_{i+2,m-1}
.$$
Thus, by induction hypothesis, we have 
$V_{i,m}\in((x+y+m-1)_{2m-1})$. 
Combining with $(\ast)$, we obtain 
$V_{i,m}\in((x+y+m)_{2m+1})$.
\end{proof}

\bigskip

\bigskip

\end{document}